\documentclass[11pt,a4paper]{article}

\usepackage{epsf,epsfig,amsfonts,amsgen,amsmath,amstext,amsbsy,amsopn,amsthm}
\usepackage{latexsym}
\usepackage{amssymb}

\usepackage{epic,eepic,color}
\usepackage{multirow}
\newtheorem{theorem}{Theorem}[section]

\newtheorem{lemma}{Lemma}[section]
\newtheorem{false statement}{False statement}

\theoremstyle{definition}

\newtheorem{claim}{Claim}

\newtheorem{remark}[claim]{Remark}

\newtheorem{problem}{Problem}
\newtheorem{case}{Case}

\begin{document}

\title{\bf Extremal problems on the Hamiltonicity of claw-free graphs}

\date{}

\author{Binlong Li\thanks{Department of Applied Mathematics,
Northwestern Polytechnical University, Xi'an, Shaanxi,~710072,~P.~R.~China.
E-mail:~{\tt libinlong@mail.nwpu.edu.cn}}~~~~
Bo Ning\thanks{Corresponding author. Center for Applied Mathematics,
Tianjin University, Tianjin, 300072,~P.~R.~China. E-mail: {\tt bo.ning@tju.edu.cn}}~~~~~
Xing Peng\thanks{Center for Applied Mathematics,
Tianjin University, Tianjin,  300072,~P.~R.~China. E-mail: {\tt x2peng@tju.edu.cn}}}

\maketitle

\begin{abstract}
In 1962, Erd\H{o}s proved that if a graph $G$ with $n$ vertices satisfies
$$
e(G)>\max\left\{\binom{n-k}{2}+k^2,\binom{\lceil(n+1)/2\rceil}{2}+\left\lfloor
\frac{n-1}{2}\right\rfloor^2\right\},
$$
where the minimum degree $\delta(G)\geq k$ and $1\leq k\leq(n-1)/2$, then
it is Hamiltonian. For $n \geq 2k+1$, let $E^k_n=K_{k}\vee (kK_1+K_{n-2k})$,
where ``$\vee$" is the ``join" operation.
One can observe  $e(E^k_n)=\binom{n-k}{2}+k^2$ and $E^k_n$
is not Hamiltonian. As $E^k_n$ contains induced claws for
$k\geq 2$, a natural question is to characterize all 2-connected
claw-free non-Hamiltonian graphs with the largest possible number
of edges. We answer this question completely by proving a claw-free
analog of Erd\H{o}s' theorem. Moreover, as byproducts, we establish
several tight spectral conditions for a 2-connected claw-free graph
to be Hamiltonian. Similar results for the traceability of connected
claw-free graphs are also obtained. Our tools include Ryj\'{a}\v{c}ek's
claw-free closure theory and Brousek's characterization of minimal
2-connected claw-free non-Hamiltonian graphs.
\end{abstract}



\section{Introduction}
Given a graph $G$, a {\it Hamilton cycle} of $G$ is a cycle which
visits all vertices of $G$. We will say that $G$ is {\it Hamiltonian}
if it contains a Hamilton cycle. Determining the Hamiltonicity of a graph
is a classically difficult problem in graph theory. An old
result due to Ore \cite{O61} states that every graph with $n$ vertices
and more than $\binom{n-1}{2}+1$ edges is Hamiltonian. Generalizing
Ore's theorem by introducing the minimum degree of a graph as a
new parameter, Erd\H{o}s \cite{E} proved the following theorem:

\begin{theorem}[Erd\H{o}s \cite{E}]
For a graph $G$  with $n$ vertices and $\delta(G)\geq k$
where $1\leq k\leq(n-1)/2$, if
$$e(G)>\max\left\{\binom{n-k}{2}+k^2,\binom{\lceil(n+1)/2\rceil}{2}+\left\lfloor
\frac{n-1}{2}\right\rfloor^2\right\}$$
then $G$ is Hamiltonian.
\end{theorem}

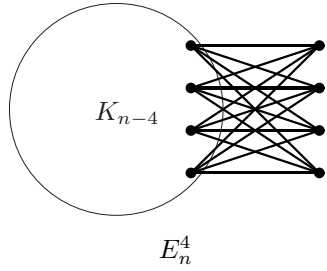
\begin{figure}[h]
\begin{center}
\setlength{\unitlength}{0.8pt} \small
\begin{picture}(535,140)
\thicklines
\put(175,0){\put(70,80){\thinlines\circle{100}}
\put(60,75){$K_{n-4}$}
\put(135,80){\multiput(-30,-30)(0,20){4}{\circle*{4}}
\multiput(30,-30)(0,20){4}{\circle*{4}}
\put(-30,-30){\line(1,0){60}} \put(-30,-30){\line(3,1){60}}
\put(-30,-30){\line(3,2){60}} \put(-30,-30){\line(1,1){60}}
\put(-30,-10){\line(3,-1){60}} \put(-30,-10){\line(1,0){60}}
\put(-30,-10){\line(3,1){60}} \put(-30,-10){\line(3,2){60}}
\put(-30,10){\line(3,1){60}} \put(-30,10){\line(1,0){60}}
\put(-30,10){\line(3,-1){60}} \put(-30,10){\line(3,-2){60}}
\put(-30,30){\line(1,0){60}} \put(-30,30){\line(3,-1){60}}
\put(-30,30){\line(3,-2){60}} \put(-30,30){\line(1,-1){60}}}
\put(90,10){$E_n^4$}}
\end{picture}
\caption{The graph $E_n^4$}
\label{graphEn4}
\end{center}
\end{figure}

Let $G_1$ and $G_2$ be two disjoint graphs. The \emph{join} of $G_1$
and $G_2$, denoted by $G_1\vee G_2$, is defined as: $V(G_1\vee G_2)=V(G_1)\cup V(G_2)$
and $E(G_1\vee G_2)=E(G_1)\cup E(G_2)\cup \{xy: x\in V(G_1),y\in V(G_2)\}$.
Erd\H{o}s' theorem is tight as shown by the following graph:
let $E^k_n:=K_{k}\vee (kK_1+K_{n-2k})$, where $n\geq 2k+1$ (see Figure \ref{graphEn4}
for an example). When $k<n/6$, we have
\[
e(E^k_n)=\binom{n-k}{2}+k^2 \geq  \binom{\lceil(n+1)/2\rceil}{2}+\left\lfloor \frac{n-1}{2}\right\rfloor^2.
\]
However, $E^k_n$ is not Hamiltonian. We say a graph $G$ is
{\it claw-free} if it does not contain $K_{1,3}$ as an
induced subgraph. We remark that $E^k_n$ is not claw-free for $k \geq 2$
and the condition $\delta(G)\geq 2$ is necessary for a graph to be Hamiltonian.

Claw-free graphs play an important role when we consider
the Hamiltonicity of graphs. A long-standing conjecture
by Matthews and Sumner \cite{MS}  asserts that every
4-connected claw-free graph is Hamiltonian. They also
constructed 3-connected claw-free graphs which are not
Hamiltonian. Therefore, it is natural to consider pairs
of forbidden subgraphs which force a 2-connected graph to be
Hamiltonian.  Bedrossian \cite{B91} solved this problem
completely by proving that if $R$ and
$S$ are connected graphs of order at least 3 with
$R,S\neq P_3$ and $G$ is  2-connected, then $G$ being
$R$-free and $S$-free implies $G$ is Hamiltonian if and
only if (up to symmetry) $R=K_{1,3}$ and
$S=P_4,P_5,P_6,C_3,Z_1,Z_2,B,N$, or $W$ (see \cite{B91}).
Recently, Bedrossian's result has received a lot of
attention. Li et al. \cite{LRWZ} and Ning and Zhang \cite{NZ}
obtained heavy subgraph versions of Bedrossian's result by
restricting Ore-type degree sum condition \cite{O60} and Fan-type 2-distance
condition \cite{F} to induced subgraphs, respectively.
Very recently, Li and Vr\'ana \cite{LV} characterized all
disconnected forbidden pairs for a 2-connected graph to be
Hamiltonian. In the other direction, Brousek \cite{B98}
characterized some important properties of  minimal 2-connected
claw-free non-Hamiltonian graphs, from which Bedrossian's result
can be obtained as a corollary.

The main goal of this paper is to give claw-free analogs of Erd\H{o}s'
theorem. We first consider the following problem:

\begin{problem}
Can we characterize all 2-connected claw-free non-Hamiltonian
graphs on $n$ vertices that have the largest number of edges?
\end{problem}

We obtain the following solution to Problem 1, and
point out that Brousek's result is a key ingredient in our proof.

\begin{theorem}\label{ThStruClawFree}
Let $G$ be a 2-connected claw-free graph on $n\geq 24$ vertices.
If $e(G)\geq e(EB_n)=e(EB'_n)$, then $G$ is Hamiltonian
unless $G=EB_n$ or $G=EB'_n$ (see Figure \ref{graphsEBnEB'n}).
\end{theorem}

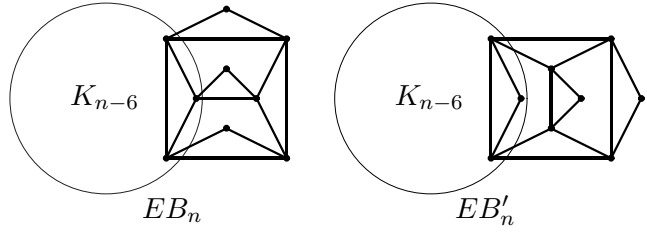
\begin{figure}[h]
\begin{center}
\setlength{\unitlength}{0.45pt}
\begin{picture}(565,205)
\thicklines

\put(0,0){\put(100,110){\thinlines\circle{160}}
\put(150,60){\circle*{4}} \put(250,60){\circle*{4}}
\put(150,160){\circle*{4}} \put(250,160){\circle*{4}}
\put(175,110){\circle*{4}} \put(200,85){\circle*{4}}
\put(225,110){\circle*{4}} \put(200,135){\circle*{4}}
\put(200,185){\circle*{4}} \put(150,60){\line(1,0){100}}
\put(150,60){\line(0,1){100}} \put(150,160){\line(1,0){100}}
\put(250,60){\line(0,1){100}} \put(150,60){\line(1,2){25}}
\put(150,160){\line(1,-2){25}} \put(150,60){\line(2,1){50}}
\put(250,60){\line(-2,1){50}} \put(250,60){\line(-1,2){25}}
\put(250,160){\line(-1,-2){25}} \put(175,110){\line(1,0){50}}
\put(175,110){\line(1,1){25}} \put(225,110){\line(-1,1){25}}
\put(150,160){\line(2,1){50}} \put(250,160){\line(-2,1){50}}
\put(70,105){$K_{n-6}$} \put(130,10){$EB_n$}}

\put(270,0){\put(100,110){\thinlines\circle{160}}
\put(150,60){\circle*{4}} \put(250,60){\circle*{4}}
\put(150,160){\circle*{4}} \put(250,160){\circle*{4}}
\put(175,110){\circle*{4}} \put(200,85){\circle*{4}}
\put(200,135){\circle*{4}} \put(225,110){\circle*{4}}
\put(275,110){\circle*{4}} \put(150,60){\line(1,0){100}}
\put(150,60){\line(0,1){100}} \put(150,160){\line(1,0){100}}
\put(250,60){\line(0,1){100}} \put(150,60){\line(1,2){25}}
\put(150,160){\line(1,-2){25}} \put(150,60){\line(2,1){50}}
\put(250,60){\line(-2,1){50}} \put(150,160){\line(2,-1){50}}
\put(250,160){\line(-2,-1){50}} \put(200,85){\line(0,1){50}}
\put(200,85){\line(1,1){25}} \put(200,135){\line(1,-1){25}}
\put(250,60){\line(1,2){25}} \put(250,160){\line(1,-2){25}}
\put(70,105){$K_{n-6}$} \put(120,10){$EB'_n$}}

\end{picture}

\caption{Graphs $EB_n$ and $EB'_n$}
\label{graphsEBnEB'n}
\end{center}
\end{figure}

Moreover, we prove a general Erd\H{o}s-type result for the
Hamiltonicity of 2-connected claw-free graphs involving
minimum degree and number of edges. We define
the graph $F_{k+1,k+1,n-2k-2}$ as: $V(G)=\bigcup_{i=1}^3V(G_i)$, where
$G_1=G_2=K_{k+1}$, $G_3=K_{n-2k-2}$, and $G_1,G_2,G_3$ are
pairwise vertex-disjoint;
$E(G)=\bigcup_{i=1}^{3}E(G_i)\bigcup (\bigcup_{1\leq i<j\leq 3}E_G(G_i,G_j))$,
where $E_G(G_i,G_j)=\{u_iu_j,v_iv_j:1\leq i<j\leq3,
u_i,v_i\in V(G_i)$ and $u_i\neq v_i$, $1\leq i\leq 3\}$.
Obviously, $\delta(F_{k+1,k+1,n-2k-2})=k$.

\begin{theorem}\label{ThgeneralErdos}\footnote{This is a solution to a problem originally appeared
in the first version of this paper, which was suggested by one referee.}
Let $k\geq 3$ and $n\geq k^2+8k+4$. Suppose that $G$ is a 2-connected
claw-free graph of order $n$ and minimum degree $\delta(G)\geq k$. If
$$e(G)\geq e(F_{k+1,k+1,n-2k-2})={n-2k-2\choose 2}+2{k+1\choose 2}+6,$$
then $G$ is Hamiltonian unless $G=F_{k+1,k+1,n-2k-2}$ (see Figure \ref{graphEpnk}).
\end{theorem}

\begin{remark}
By computation, we have ${n-2k-2\choose 2}+2{k+1\choose 2}+6<\binom{n-k}{2}+k^2$
when $n>\frac{3k}{2}+\frac{5}{2}+\frac{4}{k+2}$. Since $n\geq k^2+8k+4$,
we can see the inequality always holds. Combining
Theorem \ref{ThStruClawFree} and Theorem \ref{ThgeneralErdos},
we improve the edge condition of Erd\H{o}s' theorem for the Hamiltonicity
of 2-connected claw-free graphs. Moreover, we observe $EB'_n=F_{3,3,n-6}$.
\end{remark}

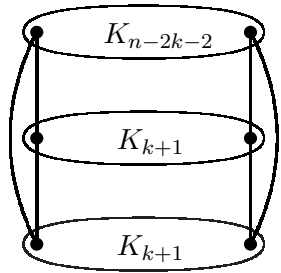
\begin{figure}[h]
\begin{center}
\begin{picture}(140,120)

\newcommand{\tuoyuan}[2]{\qbezier(#1,0)(#1,#2)(0,#2)
\qbezier(0,#2)(-#1,#2)(-#1,0) \qbezier(-#1,0)(-#1,-#2)(0,-#2)
\qbezier(0,-#2)(#1,-#2)(#1,0)}

\thicklines
\put(0,-20){\multiput(70,40)(0,40){3}{\put(0,0){\thinlines\tuoyuan{45}{10}}
\put(-40,0){\circle*{4}} \put(40,0){\circle*{4}}}
\put(60,36){$K_{k+1}$} \put(60,76){$K_{k+1}$}
\put(55,116){$K_{n-2k-2}$} \put(30,40){\line(0,1){80}}
\put(110,40){\line(0,1){80}} \qbezier(30,40)(10,80)(30,120)
\qbezier(110,40)(130,80)(110,120) }

\end{picture}
\caption{The graph $F_{k+1,k+1,n-2k-2}$} \label{graphEpnk}
\end{center}
\end{figure}

For a graph $G$, a {\it Hamilton path} of $G$ is a
path which contains all vertices of $G$. We say that a graph
is {\it traceable} if it contains a Hamilton path. We use $EN_n$ ($n\geq 6$)
to denote the graph obtained from $K_{n-3}$ by adding three disjoint pendent edges (see Figure \ref{graphENn}).
Similarly, we have the following sufficient condition
for the traceability of connected claw-free graphs.

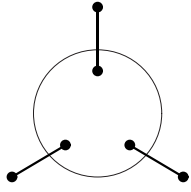
\begin{figure}[h]
\begin{center}
\setlength{\unitlength}{0.80pt}
\begin{picture}(120,100)
\put(60,40){\put(0,0){\circle{60}} \thicklines
\put(0,20){\circle*{4}} \put(0,50){\circle*{4}}
\put(-40,-30){\circle*{4}} \put(40,-30){\circle*{4}}
\put(-15,-15){\circle*{4}} \put(15,-15){\circle*{4}}
\put(0,20){\line(0,1){30}} \put(-40,-30){\line(5,3){25}}
\put(40,-30){\line(-5,3){25}}}
\end{picture}
\caption{The graph $EN_n$} \label{graphENn}
\end{center}
\end{figure}

\begin{theorem}\label{ThStruClaw-Trac}
Let $G$ be a connected claw-free graph on $n\geq 12$ vertices.
If $e(G)\geq e(EN_n)$, then $G$ is traceable
unless $G=EN_n$ (see Figure \ref{graphENn}).
\end{theorem}

In this paper, we will also prove several spectral analogs of our above
theorems. It is well known that the eigenvalues of matrices
associated with a graph can be used to describe its structure.
Thus one may ask whether we can find any spectral condition
for a graph to be Hamiltonian. Pioneer work in this direction
include Van den Heuvel's proof \cite{H95} of the famous fact
that Petersen graph is not Hamiltonian, Krivelevich
and Sudakov's result \cite{KS} for the Hamiltonicity of
$d$-regular graphs, as well as a result by Butler and
Chung \cite{BC}. In the process of finding spectral analogs of Erd\H{o}s'
theorem \cite{E}, Ning and Ge \cite{NG} first established
sufficient spectral conditions for the existence of Hamilton cycles in
graphs with $\delta(G)\geq 2$. Their research motivated
plenty of work in the similar spirit, for example, those by Feng
et al. \cite{FZL,Fetal}. Finally, Li and Ning \cite{LN16,LN17} obtained
spectral analogs of Erd\H{o}s' theorem \cite{E} and Moon-Moser's
theorem \cite{MM} on Hamilton cycles and paths of general graphs
and of bipartite graphs, respectively.

For a graph $G$,  let $A$ be the adjacency matrix of $G$ and $D$
the diagonal matrix of degrees.  The \emph{signless Laplacian
matrix} $Q$ of $G$ is defined as $D+A$. The largest eigenvalue of
$A$ (resp. $Q$) is denoted by $\mu(G)$ (resp. $q(G)$). We will
use $\overline G$ to denote the {\it complement} of $G$.

We prove the following sufficient spectral conditions for the
Hamiltonicity of 2-connected claw-free graphs.

\begin{theorem}\label{ThmuG}
Let $G$ be a 2-connected claw-free graph on $n$ vertices.
If $n\geq 30$ and $\mu(G)\geq\mu(EB_n)$, then $G$ is Hamiltonian
unless $G=EB_n$ (see Figure \ref{graphsEBnEB'n}).
\end{theorem}

\begin{theorem}\label{ThmuCG}
Let $G$ be a 2-connected claw-free graph on $n$ vertices.
If $n\geq 219$ and $\mu(\overline{G})\leq\mu(\overline{EB'_n})$,
then $G$ is Hamiltonian unless $G=EB'_n$ (see Figure \ref{graphsEBnEB'n}).
\end{theorem}

Similarly, we prove the following theorem involving $q(G)$.
\begin{theorem}\label{ThqG}
Let $G$ be a 2-connected claw-free graph on $n$ vertices.
If $n\geq 51$ and $q(G)\geq q(EB_n)$, then $G$ is Hamiltonian unless
$G=EB_n$ (see Figure \ref{graphsEBnEB'n}).
\end{theorem}

\begin{remark}
We would like to point out the following interesting observation.
Theorem \ref{ThStruClawFree} involves extremal graphs $EB_n$
and $EB'_n$. These two graphs have the same number of edges.
To prove Theorems \ref{ThmuG} and \ref{ThqG}, we have
to compare the spectral radii and signless spectral radii of
these two graphs, but it turns out that their values are
different. This means that problems in extremal graph theory
and their spectral analogues are not completely the same!
\end{remark}

We also prove the following spectral analog of Theorem \ref{ThStruClaw-Trac}.
\begin{theorem}\label{ThTraceable}
Let $G$ be a connected claw-free graph on $n\geq 18$ vertices. If
$q(G)\geq q(EN_n)$, then $G$ is traceable unless $G=EN_n$ (see Figure \ref{graphENn}).
\end{theorem}

We follow standard notation throughout this paper. For
those not defined here, we refer the reader to monographs \cite{BM,CDS}.
Let $G$ be a graph and $v$ be a vertex of $V(G)$. We denote
by $N_G(v)$ the set of vertices which are adjacent to $v$ in $G$.
If $H$ is a subgraph of $G$, then let $N_H(v)=N_G(v)\cap V(H)$.
Set $d_H(v)=|N_H(v)|$. When the graph $G$ is clear from the context,
we write $N(v)$ and $d(v)$ for $N_G(v)$ and $d_G(v)$, respectively.
The \emph{minimum degree} of $G$, denoted by $\delta(G)$, equals to
$\min\{d(v):v\in V(G)\}$. Let $S\subset V(G)$. We denote by $G[S]$
the subgraph of $G$ induced by $S$ and by $G-S$ the subgraph
of $G$ induced by $V(G)\backslash V(S)$. We use $e(G)$
to denote the number of edges in $G$, and $\omega(G)$
to denote the clique number of $G$. For two graphs
$G_1$ and $G_2$, let $G_1+G_2$ and $G_1\vee G_2$
denote the \emph{disjoint union} and the \emph{join}
of $G_1$ and $G_2$, respectively. A graph is called
\emph{nonseparable} if it is connected and has no cut-vertex. Following
the terminology in \cite{F74}, we say that a graph is
a \emph{block-chain} if it is nonseparable or it has
connectivity 1 and has exactly two end-blocks.
A graph $G$ is called \emph{Hamiltonian-connected} if for
each pair of vertices $x,y\in V(G)$, there is a Hamilton path from
$x$ to $y$ in $G$.

The rest of this  paper is organized as follows. In Section
2, we will recall a few theorems related to claw-free closure
theory and prove several structural lemmas. In Section 3,
we will present proofs of Theorems \ref{ThStruClawFree},
\ref{ThgeneralErdos}, and  \ref{ThStruClaw-Trac}. We
will establish several spectral inequalities  in Section 4. We will prove
Theorems \ref{ThmuG}, \ref{ThmuCG}, \ref{ThqG},
and \ref{ThTraceable} in Section 5.  We will mention a concluding
remark and a problem for future work in the last section.
\section{Preliminaries}
In this section, we will list several theorems from
structural graph theory and prove a number of useful lemmas.
We first recall the claw-free closure theory introduced by
Ryj\'{a}\v{c}ek \cite{R}. For completeness, we include necessary
definitions here. For more information, please see \cite{R}.

Let $G$ be a claw-free graph. For a vertex $x\in V(G)$, if the
neighborhood of $x$ induces a connected but not complete subgraph
of $G$, then $x$ is called an \emph{eligible} vertex in $G$. Set
$B_G(x)=\{uv: u,v\in N(x), uv\notin E(G)\}$. Let $G'_x$ be a new
graph such that  $V(G'_x)=V(G)$ and $E(G'_x)=E(G)\cup B_G(x)$. We
call $G'_x$ the \emph{local completion of $G$ at $x$}. The \emph{closure}
of $G$, denoted by $cl(G)$, is defined by a sequence of graphs
$G_1,G_2,\ldots,G_t$  and vertices $x_1,x_2,\ldots,x_{t-1}$ such that:
\begin{description}
\item[(a)]  $G_1=G$ and $G_t=cl(G)$;
\item[(b)]  $x_i$ is an eligible vertex of $G_i$, and $G_{i+1}=(G_i)'_{x_i}$,
$1\leq i\leq t-1$;
\item[(c)]  $cl(G)$ has no eligible vertices.
\end{description}

The following theorems are very useful when we study
Hamiltonian properties of claw-free graphs.

\begin{theorem}[Ryj\'{a}\v{c}ek \cite{R}]\label{ThR}
Let $G$ be a claw-free graph. Then $G$ is Hamiltonian if and only if
$cl(G)$ is Hamiltonian.
\end{theorem}

\begin{theorem}[Brandt, Favaro, and Ryj\'{a}\v{c}ek \cite{BFR}] \label{ThBFR}
Let $G$ be a claw-free graph. Then $G$ is traceable if and only if
$cl(G)$ is traceable.
\end{theorem}

\begin{theorem} [Duffus, Jacboson, and Gould \cite{DJG}]\label{ThDuJaGo}
Let $G$ be a claw-free graph with no induced copies of $EN_6$ (see Figure \ref{graphENn}).\\
(1) If $G$ is connected, then $G$ is traceable.\\
(2) If $G$ is 2-connected, then $G$ is Hamiltonian.
\end{theorem}

We also need a theorem of Brousek \cite{B98}.
First, let us recall some notation by Brousek \cite{B98}.
Let $\mathcal{P}$ denote the class of graphs obtained from two
triangles $a_1a_2a_3a_1$ and $b_1b_2b_3b_1$  such that each
pair $\{a_i,b_i\}$ is connected by a triangle or a path of
$k_i\geq 3$ vertices (see Figure \ref{clawfreenonHam} for some examples). We use
$P_{x_1,x_2,x_3}$ to denote the graph from $\mathcal{P}$,
where $x_i=T$ if $\{a_i,b_i\}$ is connected by a triangle,
and $x_i=k_i$ if $\{a_i,b_i\}$ is connected by a path with
$k_i$ vertices. For each graph from the collection
$\{P_{T,T,T}, P_{3,T,T}, P_{3,3,T},P_{3,3,3}\}$
(see Figure \ref{clawfreenonHam}), consider replacing one of its triangles by $K_{n-6}$,
and let $\mathcal{EB}_n$ ($n\geq 9$) be the collection of all graphs obtained
in this way. For example, if we replace one triangle in $P_{T,T,T}$
by $K_{n-6}$, then we get one of two new graphs which are denoted by
$EB_n$ and $EB'_n$, respectively (see Figure \ref{graphsEBnEB'n}). It is important
to notice that each graph in $\mathcal{EB}_n$ is a subgraph
of $EB_n$ or $EB'_n$.
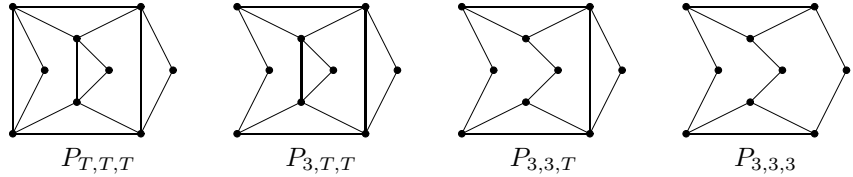
\begin{figure}[h]
\begin{center}
\setlength{\unitlength}{0.6pt}\small
\begin{picture}(560,120)\label{claw9}
\put(0,0){\multiput(20,30)(80,0){2}{\put(0,0){\circle*{4}}
\put(0,80){\circle*{4}} \put(20,40){\circle*{4}}
\put(0,0){\line(0,1){80}} \put(0,0){\line(1,2){20}}
\put(0,80){\line(1,-2){20}}} \put(20,30){\line(1,0){80}}
\put(20,30){\line(2,1){40}} \put(100,30){\line(-2,1){40}}
\put(20,110){\line(1,0){80}} \put(20,110){\line(2,-1){40}}
\put(100,110){\line(-2,-1){40}} \put(60,50){\circle*{4}}
\put(60,90){\circle*{4}} \put(80,70){\circle*{4}}
\put(60,50){\line(0,1){40}} \put(60,50){\line(1,1){20}}
\put(60,90){\line(1,-1){20}} \put(50,10){$P_{T,T,T}$}}

\put(140,0){\multiput(20,30)(80,0){2}{\put(0,0){\circle*{4}}
\put(0,80){\circle*{4}} \put(20,40){\circle*{4}}
\put(0,0){\line(1,2){20}} \put(0,80){\line(1,-2){20}}}
\put(20,30){\line(1,0){80}} \put(100,30){\line(0,1){80}}
\put(20,30){\line(2,1){40}} \put(100,30){\line(-2,1){40}}
\put(20,110){\line(1,0){80}} \put(20,110){\line(2,-1){40}}
\put(100,110){\line(-2,-1){40}} \put(60,50){\circle*{4}}
\put(60,90){\circle*{4}} \put(80,70){\circle*{4}}
\put(60,50){\line(0,1){40}} \put(60,50){\line(1,1){20}}
\put(60,90){\line(1,-1){20}} \put(50,10){$P_{3,T,T}$}}

\put(280,0){\multiput(20,30)(80,0){2}{\put(0,0){\circle*{4}}
\put(0,80){\circle*{4}} \put(20,40){\circle*{4}}
\put(0,0){\line(1,2){20}} \put(0,80){\line(1,-2){20}}}
\put(20,30){\line(1,0){80}} \put(100,30){\line(0,1){80}}
\put(20,30){\line(2,1){40}} \put(100,30){\line(-2,1){40}}
\put(20,110){\line(1,0){80}} \put(20,110){\line(2,-1){40}}
\put(100,110){\line(-2,-1){40}} \put(60,50){\circle*{4}}
\put(60,90){\circle*{4}} \put(80,70){\circle*{4}}
\put(60,50){\line(1,1){20}} \put(60,90){\line(1,-1){20}}
\put(50,10){$P_{3,3,T}$}}

\put(420,0){\multiput(20,30)(80,0){2}{\put(0,0){\circle*{4}}
\put(0,80){\circle*{4}} \put(20,40){\circle*{4}}
\put(0,0){\line(1,2){20}} \put(0,80){\line(1,-2){20}}}
\put(20,30){\line(1,0){80}} \put(20,30){\line(2,1){40}}
\put(100,30){\line(-2,1){40}} \put(20,110){\line(1,0){80}}
\put(20,110){\line(2,-1){40}} \put(100,110){\line(-2,-1){40}}
\put(60,50){\circle*{4}} \put(60,90){\circle*{4}}
\put(80,70){\circle*{4}} \put(60,50){\line(1,1){20}}
\put(60,90){\line(1,-1){20}} \put(50,10){$P_{3,3,3}$}}
\end{picture}

\caption{Four 2-connected claw-free non-Hamiltonian graphs of
order 9}
\label{clawfreenonHam}
\end{center}
\end{figure}

A claw-free graph is said to be \emph{closed} if $cl(G)=G$.
It is not difficult to see that, for every vertex $x$ of a closed claw-free
graph $G$, the neighborhood $N(x)$ is either a clique or the disjoint
union of two cliques in $G$ (see \cite{R}).

\begin{theorem}[Brousek \cite{B98}]\label{ThB98}
Every 2-connected claw-free non-Hamiltonian graph contains an
induced subgraph isomorphic to a graph in $\mathcal{P}$.
\end{theorem}

Let us prove the following lemma on the Hamiltonicity of
closed claw-free graphs.

\setcounter{lemma}{4}
\begin{lemma}\label{Ledudv}
Let $G$ be a closed claw-free graph on $n$ vertices. If
there are two nonadjacent vertices $u,v\in V(G)$ such that
$d(u)+d(v)\geq n$, then $G$ is Hamiltonian.
\end{lemma}

\begin{proof}
Since $uv\notin E(G)$ and $d(u)+d(v)\geq n$, $u$ and $v$ have at
least two common neighbors. Since $G$ is claw-free, if there exist two common
neighbors of $u$ and $v$ which are adjacent, then they would be eligible.  Thus $u$
and $v$ must be adjacent in $G$, which is a contradiction. Therefore, each two common neighbors
of $u$ and $v$ are nonadjacent. If $u$ and $v$ have three common neighbors, then
there will be a claw with the center $u$. This implies that $u$ and $v$
have exactly two common neighbors, say $x$ and $y$. Since $|N(u)\cap N(v)|=2$,
we get $|N(u)\cup N(v)|\geq n-2$ and
$N(u)\cup N(v)\cup \{u,v\}\subseteq V(G)$. It follows $|N(u)\cup N(v)|\leq n-2$,
and so $N(u)\cup N(v)\cup \{u,v\}=V(G)$. This means that every vertex in
$V(G)\backslash\{u,v,x,y\}$ is adjacent to $u$ or $v$. Note that each of $N(u)$
and $N(v)$ is a disjoint union of two cliques. This implies that $G$ consists
of four cliques $C_1,C_2,C_3$ and $C_4$ containing $\{u,x\},\{x,v\},\{v,y\}$,
and $\{y,u\}$, respectively. It is easy to check that $G$ has a Hamilton cycle.
The proof is complete.
\end{proof}

The second lemma concerns the clique number of a closed claw-free non-Hamiltonian
graph with a give number of edges.
\begin{lemma}\label{LeG}
Let $k$ be a positive integer and $G$ be a closed claw-free
non-Hamiltonian graph on $n\geq 2k+2$ vertices. If $$e(G)\geq
\binom{n-k-1}{2}+\binom{k+2}{2}+1$$ then $\omega(G)\geq n-k$.
\end{lemma}
\begin{proof}
A vertex $v$ is a \emph{heavy vertex} of $G$ if $d(v)\geq n/2$. By
Lemma \ref{Ledudv}, every two heavy vertices are adjacent in $G$. Let
$T$ be a maximum clique of $G$ such that all heavy vertices of $G$ are contained
in $T$ and let $H=G-V(T)$. Set $t=|V(T)|$.

Pick an arbitrary vertex  $v \in H$.  We know $v$ is nonadjacent to at
least one vertex in $T$. In fact, if $v$ has at least two neighbors in $T$,
say $x$ and $y$, then for each vertex $z\in V(T)\backslash \{x,y\}$,
there is a path from $v$ to $z$ in $G[N_T(x)\cup \{v\}]$. Since $G$ is closed,
$v$ is adjacent to $z$ in $G$. Now $G[V(T)\cup \{v\}]$
is a clique of $G$ which contains all heavy vertices but with more
vertices than $T$, a contradiction. This implies that every
vertex in $H$ has at most one neighbor in $T$.

We first assume $1\leq t\leq (n+1)/2$. Recall that every vertex
in $H$ has degree at most $(n-1)/2$ in $G$.  Therefore
 $$\sum_{v\in
V(H)}(d_T(v)+d(v))\leq(n-t)\left(\frac{n-1}{2}+1\right)=\frac{(n-t)(n+1)}{2}.$$
Moreover, by using calculus, we can obtain
\begin{align*}
e(G)    & =e(G[T])+\frac{\sum_{v\in V(H)}(d_T(v)+d(v))}{2}\\
        & \leq\binom{t}{2}+\frac{(n-t)(n+1)}{4}\\
        & =\frac{1}{2}t^2-\frac{n+3}{4}t+\frac{n(n+1)}{4}\\
        & \leq\frac{n^2-1}{4}\\
        & \leq\binom{n-k-1}{2}+\binom{k+2}{2}\\
        & <e(G),
\end{align*}
a contradiction.

Now we assume  $n/2+1\leq t\leq n-k-1$. Recall that every vertex
in $H$ has at most one neighbor in $T$. Thus, by using calculus,
we can obtain
\begin{align*}
e(G)    & =e(G[T])+e(H)+e(T,V(H))\\
        & \leq\binom{t}{2}+\binom{n-t}{2}+(n-t)\\
        & =\binom{t}{2}+\binom{n-t+1}{2}\\
        & \leq\binom{n-k-1}{2}+\binom{k+2}{2}\\
        &<e(G),
\end{align*}
a contradiction.

So we have $t\geq n-k$. This implies that $\omega(G)\geq t\geq n-k$. The proof is complete.
\end{proof}

\begin{remark}
Let $G_1=K_1\vee (K_{n-k-2}+K_{k+1})$, where $n\geq 2k+2$. One can find that $G$ is a closed claw-free
graph which is non-Hamiltonian. Notice that $e(G_1)=\binom{n-k-1}{2}+\binom{k+2}{2}$ and $\omega(G)=n-k-1$.
This example shows that the edge condition in Lemma \ref{LeG} is tight.
\end{remark}

The next two lemmas give us  characterizations of closed claw-free non-Hamiltonian
graphs under some assumptions on the clique number and the number of edges,
respectively.
\begin{lemma}\label{LeEBn}
Let $G$ be a 2-connected closed claw-free non-Hamiltonian graph on
$n$ vertices. If $\omega(G)\geq n-6$, then $G\in\mathcal{EB}_n$.
\end{lemma}

\begin{proof}
Since $G$ is a 2-connected graph which is closed, claw-free, and not Hamiltonian,
Theorem \ref{ThB98} implies that $G$ contains an induced subgraph $H$
isomorphic to a graph in $\mathcal{P}$. Let $C$ be a maximum clique
of $G$. If $|V(C)\cap V(H)|\leq 2$, then $|V(C)\cup V(H)|\geq n-6+9-2=n+1$, a
contradiction. Hence $|V(C)\cap V(H)|\geq 3$. Since $\omega(H)=3$, we get
$|V(C)\cap V(H)|\leq 3$. Thus $|V(C)\cap V(H)|=3$. If $|V(H)|\geq 10$,
then $|V(C)\cup V(H)|\geq n-6+10-3=n+1$, a contradiction. If $\omega(G)\geq n-5$,
then $|V(C)\cup V(H)|\geq n-5+9-3=n+1$, a contradiction. Therefore,
$|V(H)|=9$, $\omega(G)=n-6$, and $G$  contains a graph in $\mathcal{EB}_n$ as a subgraph.
Furthermore, if $G\notin \mathcal{EB}_n$ then $G$ is Hamiltonian, a contradiction.
This implies that $G\in \mathcal{EB}_n$ and this completes the proof.
\end{proof}

\begin{lemma}\label{LeeGEBn}
Let $G$ be a 2-connected claw-free graph on $n\geq 14$ vertices. If
$e(G)\geq n(n-15)/2+57$, then $G$ is Hamiltonian unless $G\subseteq
EB_n$ or $G \subseteq EB'_n$.
\end{lemma}

\begin{proof}
Let $G'=cl(G)$. If $G'$ is Hamiltonian, then $G$ is also Hamiltonian
by Theorem \ref{ThR}. Now we assume  $G'$ is non-Hamiltonian. Clearly
$e(G')\geq e(G)\geq n(n-15)/2+57$. By Lemma \ref{LeG}, we get
$\omega(G)\geq n-6$.  Lemma \ref{LeEBn} gives $G'\in\mathcal{EB}_n$.
This implies that $G$ is a subgraph of a graph in $\mathcal{EB}_n$.
Note that every graph in $\mathcal{EB}_n$ is either a subgraph of
$EB_n$ or a subgraph of $EB'_n$. Thus either $G\subseteq EB_n$ or
$G \subseteq EB'_n$.
\end{proof}

The following two lemmas are in the same spirit as Lemmas \ref{LeEBn}
and \ref{LeeGEBn}, respectively.
\begin{lemma}\label{LeENn}
Let $G$ be a connected closed claw-free non-traceable graph on $n$
vertices. If $\omega(G)\geq n-3$ then $G=EN_n$.
\end{lemma}

\begin{proof}
Since $G$ is connected claw-free and non-traceable, Theorem
\ref{ThDuJaGo} implies that $G$ contains an induced subgraph $H$
isomorphic to $EN_6$. Let $C$ be a maximum clique of $G$. If
$|V(C)\cap V(H)|\leq 2$, then $|V(C)\cup V(H)|\geq n-3+6-2=n+1$,
a contradiction. Hence $|V(C)\cap V(H)|\geq 3$. Since $\omega(H)=3$,
$|V(C)\cap V(H)|\leq 3$, and thus $|V(C)\cap V(H)|=3$ and $C$
contains the (unique) triangle of $H$.  It follows $EN_n\subseteq G$.
Note that the graph obtained from $EN_n$ by adding at least one
more edge is traceable. Thus $G=EN_n$.
\end{proof}

\begin{lemma}\label{LeeGENn}
Let $G$ be a connected claw-free graph on $n\geq 8$ vertices. If
$e(G)\geq n(n-9)/2+21$, then $G$ is traceable unless $G\subseteq
EN_n$.
\end{lemma}

\begin{proof}
Suppose  $G$ is not traceable. Let $G'=cl(G)$.  By Theorem \ref{ThBFR},
$G'$ is not traceable, and hence is not Hamiltonian. Clearly, $e(G')\geq e(G)\geq n(n-9)/2+21$. Since
$G'$ is closed, we have $\omega(G)\geq n-3$ by Lemma \ref{LeG}.
 Lemma \ref{LeENn} implies $G'=EN_n$. Thus $G\subseteq EN_n$.
\end{proof}

\section{Proofs of Theorems \ref{ThStruClawFree},
\ref{ThgeneralErdos} and \ref{ThStruClaw-Trac}}
In this section, we prove Theorems \ref{ThStruClawFree},
\ref{ThgeneralErdos}, and \ref{ThStruClaw-Trac}, respectively.

\smallskip
\noindent
{\bf Proof of Theorem \ref{ThStruClawFree}.} Suppose $G$ is not Hamiltonian.
Since $$e(G)\geq e(EB_n)=\binom{n-6}{2}+12=\frac{n^2-13n}{2}+33\geq \frac{n^2-15n}{2}+57$$
when $n\geq 24$, either $G\subseteq EB_n$ or $G \subseteq EB'_n$ by
Lemma \ref{LeeGEBn}. As $e(G)\geq e(EB_n)$, we have either $G=EB_n$
or $G=EB'_n$. {\hfill$\Box$}
\smallskip

The proof of Theorem \ref{ThgeneralErdos} needs the following theorems on Hamiltonian properties of graphs.

\begin{theorem}[Ore \cite{O63}]\label{ThO63}
Let $G$ be a graph on $n$ vertices. If each pair of nonadjacent vertices has degree sum
at least $n+1$, then $G$ is Hamiltonian-connected.
\end{theorem}

\begin{theorem}[Benhocine and Wojda \cite{BW87}]\label{ThBW}
 Let $G$ be a 3-connected
graph of order $n\geq 4$. If $\max\{d(u),d(v)\}\geq(n+1)/2$ for
every pair of vertices $u,v$ with distance 2, then $G$ is
Hamiltonian-connected.
\end{theorem}

\begin{theorem}[Matthews and Sumner \cite{MS}]\label{ThMS}
Let $G$ be a 2-connected claw-free graph on $n$ vertices.
If $\delta(G)\geq \frac{n-2}{3}$, then $G$ is Hamiltonian.
\end{theorem}

\begin{theorem}[Li \cite{L95}]\label{ThLi}
If $G$ is a 2-connected claw-free graph of order $n$ with
$\delta(G)\geq n/4$, then $G$ is Hamiltonian or $G\in\mathcal{F}$,
where $\mathcal{F}$ is the set of all the graphs defined as follows:
$G$ is in $\mathcal{F}$ if it can be decomposed into three
vertex-disjoint subgraphs $G_1$, $G_2$, and $G_3$ such that
$V(G)=\bigcup_{i=1}^3V(G_i)$ and $E_G(G_i,G_j)=\{u_iu_j,v_iv_j:
u_i,v_i\in V(G_i),u_i\neq v_i,1\leq i<j\leq3\}$.
\end{theorem}

\smallskip
\noindent
{\bf Proof of Theorem \ref{ThgeneralErdos}.}
We prove Theorem \ref{ThgeneralErdos} by contradiction. Let $G$ be a
counterexample to Theorem \ref{ThgeneralErdos} with the maximum number
of edges (depending on $n$ and $k$). Clearly $G$ is closed. Suppose
that $G$ is not Hamiltonian. Notice
$$e(G)\geq{n-2k-2\choose 2}+2{k+1\choose 2}+6\geq{n-2k-3\choose 2}+{2k+4\choose 2}+1$$
when $n\geq \max\{k^2+8k+4,4k+6\}=k^2+8k+4$. By Lemma \ref{LeG}, $\omega(G)\geq n-2k-2$. In the following,
a vertex $x$ is called \emph{simple} if $N(x)$ is a clique in $G$.

\setcounter{claim}{0}
\begin{claim}
  $\alpha(G)\geq 4$.
\end{claim}

\begin{proof}
Recall that Chv\'atal-Erd\H{o}s Theorem \cite{CE}
states that every graph $G$ is Hamiltonian if $\alpha(G)\leq \kappa(G)$,
where $\alpha(G)$ and $\kappa(G)$ are the independence
number and connectivity of $G$, respectively.
If $\alpha(G)\leq 2$, then $G$ is Hamiltonian by
Chv\'atal-Erd\H{o}s Theorem \cite{CE}, a contradiction.

Suppose that $\alpha(G)=3$. By Theorem \ref{ThB98}, $G$
contains $P=P_{x_1,x_2,x_3}$ as an induced
subgraph. It is clear that $P=P_{T,T,T}$, since otherwise
$\alpha(G)\geq 4$. Let $K_i$ be the maximal clique containing
$\{a_i,b_i,c_i\}$, $i=1,2,3$, where $a_ib_ic_ia_i$'s are three
vertex-disjoint triangles in $T$ (see Figure \ref{clawfreenonHam} in Section 2).
We shall show that every vertex of $G$ is contained in
$K_1\cup K_2\cup K_3$. Indeed, if there exists a vertex, say,
$v\in V(G)\backslash (K_1\cup K_2\cup K_3)$, then
$v$ is adjacent to at most one vertex in $K_i$ for $i\in \{1,2,3\}$; otherwise,
$v$ is adjacent to every vertex in $K_i$ since $G$ is closed, and it
contradicts the maximality of the choice of $K_i$. Consider $G[\{v,c_1,c_2,c_3\}]$.
Since $c_1,c_2,c_3$ are independent and $\alpha(G)=3$, without
loss of generality, we can assume that $vc_1\in E(G)$. Next, consider
the graph $G[\{v,b_1,c_2,c_3\}]$. Notice that $b_1,c_2,c_3$ are independent
vertices. Since $vc_1\in E(G)$, $vb_1\notin E(G)$. Without loss of
generality, assume that $vc_2\in E(G)$. Finally, consider $G[\{v,b_1,a_2,c_3\}]$.
We can deduce $vc_3\in E(G)$. Now $\{v,c_1,c_2,c_3\}$ induces a claw,
a contradiction. This shows that every vertex of $G$ is contained in $K_1\cup K_2\cup K_3$.

Moreover, one can see if there is an edge in
$E':=\bigcup_{1\leq i<j\leq 3}E(K_i,K_j)\backslash E(P)$
then $G$ is Hamiltonian, a contradiction. So
$E'=\emptyset$. Recall that $\omega(G)\geq
n-2k-2$ and $\delta(G)\geq k$. We conclude that (up to symmetry)
$|K_1|=n-2k-2$, $|K_2|=|K_3|=k+1$, and $G=F_{k+1,k+1,n-2k-2}$ (see
Figure \ref{graphEpnk}).
\end{proof}

Now let $K$ be a maximum clique of $G$, and let $G'=G-V(K)$. So
$|V(G')|\leq 2k+2$. Recall that in $G$, the neighborhood of each
vertex is a clique or two disjoint cliques. Moreover, we can
easily see that $N_{G'}(x)$ is a clique for every vertex
$x\in N_K(G')$. (Note that every vertex
of $G'$ has at most one neighbor in $K$.)

Clearly $\delta(G')\geq k-1$.
If $G'$ has only one or two vertices, then $G$ is
Hamiltonian. So assume that $|V(G')|\geq 3$. For a
component $H$ of $G'$, we call a path $P$ a
\emph{perfect path} corresponding to $H$, if $P$ connects two
vertices in $K$ and its internal vertex set is $V(H)$.

We divide the left part into three cases.
\begin{case}
  $G'$ is disconnected.
\end{case}

Since $\delta(G')\geq k-1$ and $|V(G')|\leq 2k+2$, each
component of $G'$ has order at least $k$. This implies that $G'$ has
exactly two components, and each component has order at most $k+2$.
Let $H_i$, $i=1,2$, be the two components of $G'$. We now prove that
there is a perfect path $P_i$ corresponding to $H_i$ for $i=1,2$.

We have $\delta(H_i)\geq k-1\geq \frac{|V(H_i)|+1}{2}$,
unless $(k,|V(H_i)|)\in\{(3,4),(3,5),(4,6)\}$. Suppose
$(k,|V(H_i)|)\notin\{(3,4),(3,5),(4,6)\}$. By Theorem \ref{ThO63}, $H_i$
is Hamilton-connected. Since $G$ is 2-connected, there is a perfect
path $P_i$ corresponding to $H_i$. Now, we consider the leftover case
of $(k,|V(H_i)|)\in\{(3,4),(3,5),(4,6)\}$. Recall that $H_i$ is
closed, connected, and claw-free. One can check that $H_i$ is
one of the graphs in Figure \ref{figureH_i}. Moreover, every vertex
of degree $k-1$ in $H_i$ has a neighbor in $K$. It is easy to see that there is a
perfect path $P_i$ corresponding to $H_i$.

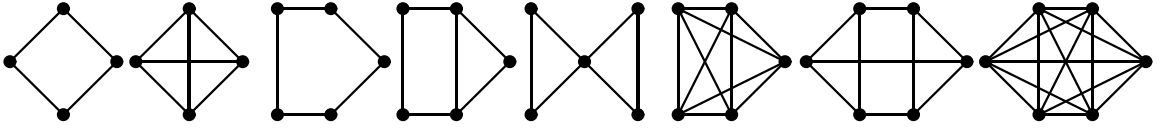
\begin{figure}[h]
\begin{center}
\begin{picture}(485,70)

\thicklines

\put(0,0){\put(10,30){\circle*{4}} \put(30,10){\circle*{4}}
\put(50,30){\circle*{4}} \put(30,50){\circle*{4}}
\put(10,30){\line(1,-1){20}} \put(10,30){\line(1,1){20}}
\put(50,30){\line(-1,-1){20}} \put(50,30){\line(-1,1){20}}}

\put(47,0){\put(10,30){\circle*{4}} \put(30,10){\circle*{4}}
\put(50,30){\circle*{4}} \put(30,50){\circle*{4}}
\put(10,30){\line(1,-1){20}} \put(10,30){\line(1,0){40}}
\put(10,30){\line(1,1){20}} \put(50,30){\line(-1,-1){20}}
\put(50,30){\line(-1,1){20}} \put(30,10){\line(0,1){40}}}

\put(100,0){\put(10,10){\circle*{4}} \put(10,50){\circle*{4}}
\put(30,10){\circle*{4}} \put(30,50){\circle*{4}}
\put(50,30){\circle*{4}} \put(10,10){\line(1,0){20}}
\put(10,10){\line(0,1){40}} \put(10,50){\line(1,0){20}}
\put(50,30){\line(-1,-1){20}} \put(50,30){\line(-1,1){20}}}

\put(147,0){\put(10,10){\circle*{4}} \put(10,50){\circle*{4}}
\put(30,10){\circle*{4}} \put(30,50){\circle*{4}}
\put(50,30){\circle*{4}} \put(10,10){\line(1,0){20}}
\put(10,10){\line(0,1){40}} \put(10,50){\line(1,0){20}}
\put(30,10){\line(0,1){40}} \put(50,30){\line(-1,-1){20}}
\put(50,30){\line(-1,1){20}}}

\put(195,0){\put(10,10){\circle*{4}} \put(10,50){\circle*{4}}
\put(30,30){\circle*{4}} \put(50,10){\circle*{4}}
\put(50,50){\circle*{4}} \put(10,10){\line(1,1){40}}
\put(10,10){\line(0,1){40}} \put(10,50){\line(1,-1){40}}
\put(50,10){\line(0,1){40}}}

\put(250,0){\put(10,10){\circle*{4}} \put(10,50){\circle*{4}}
\put(30,10){\circle*{4}} \put(30,50){\circle*{4}}
\put(50,30){\circle*{4}} \put(10,10){\line(1,0){20}}
\put(10,10){\line(2,1){40}} \put(10,10){\line(1,2){20}}
\put(10,10){\line(0,1){40}} \put(10,50){\line(2,-1){40}}
\put(10,50){\line(1,-2){20}} \put(10,50){\line(1,0){20}}
\put(30,10){\line(0,1){40}} \put(50,30){\line(-1,-1){20}}
\put(50,30){\line(-1,1){20}}}

\put(298,0){\put(10,30){\circle*{4}} \put(30,10){\circle*{4}}
\put(30,50){\circle*{4}} \put(50,10){\circle*{4}}
\put(50,50){\circle*{4}} \put(70,30){\circle*{4}}
\put(10,30){\line(1,-1){20}} \put(10,30){\line(1,0){60}}
\put(10,30){\line(1,1){20}} \put(30,10){\line(1,0){20}}
\put(30,10){\line(0,1){40}} \put(30,50){\line(1,0){20}}
\put(50,10){\line(0,1){40}} \put(70,30){\line(-1,-1){20}}
\put(70,30){\line(-1,1){20}}}

\put(365,0){\put(10,30){\circle*{4}} \put(30,10){\circle*{4}}
\put(30,50){\circle*{4}} \put(50,10){\circle*{4}}
\put(50,50){\circle*{4}} \put(70,30){\circle*{4}}
\put(10,30){\line(1,-1){20}} \put(10,30){\line(2,-1){40}}
\put(10,30){\line(1,0){60}} \put(10,30){\line(2,1){40}}
\put(10,30){\line(1,1){20}} \put(30,10){\line(1,0){20}}
\put(30,10){\line(2,1){40}} \put(30,10){\line(1,2){20}}
\put(30,10){\line(0,1){40}} \put(30,50){\line(2,-1){40}}
\put(30,50){\line(1,-2){20}} \put(30,50){\line(1,0){20}}
\put(50,10){\line(0,1){40}} \put(70,30){\line(-1,-1){20}}
\put(70,30){\line(-1,1){20}}}

\end{picture}
\caption{All possible constructions of the subgraph $H_i$}
\label{figureH_i}
\end{center}
\end{figure}
Since $P_i$ is a perfect path corresponding to $H_i$ ($i=1,2$)
where $H_1$ and $H_2$ are disjoint,
$(V(P_1)\cap V(H_1))\cap (V(P_2)\cap V(H_2))=\emptyset$.
Furthermore, as $N_{G'}(x)$ is a clique for each $x\in N_K(G')$,
the end-vertices of $P_1$ in $K$ are different
from the ones of $P_2$ in $K$. Now we can find two
disjoint paths in $G[K]$, which together with $P_1,P_2$,
form a Hamilton cycle of $G$, a contradiction.

\begin{case}
The connectivity of $G'$ is 1.
\end{case}

We claim that $G'$ has exactly two end-blocks. Suppose not. Since
$G'$ is claw-free, each cut-vertex of $G'$ is contained in exactly two blocks.
Hence, $G'$ has a block containing three cut-vertices, and $G'$
has three end-blocks that are pairwise vertex-disjoint. Since
$\delta(G')\geq k-1$, each end-block has order at least $k$, and
thus $|V(G')|\geq 3k>2k+2$, a contradiction. So $G'$ is a block-chain.

Let $B_i$ be the two end-blocks of $G'$, and let $c_i$ be the cut vertex of $G'$
contained in $B_i$, $i=1,2$ (possibly $c_1=c_2$). Note that each
end-block has order at least $k$ and at most $k+3$. Moreover, there
are at most 2 vertices in $V(G'-(V(B_1)\cup V(B_2)))$.

First consider the graph $H=G'-(V(B_1-c_1)\cup V(B_2-c_2))$. Note that
$|V(H)|\leq 4$ and $H$ is claw-free and closed. One can check that $H$
is either a path of order at most 4 between $c_1$ and $c_2$, or one of the
graphs in Figure \ref{figureH}. Thus $H$ has a Hamilton path between $c_1$ and
$c_2$.

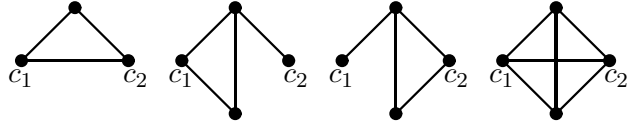
\begin{figure}[h]
\begin{center}
\begin{picture}(290,70)

\thicklines

\put(10,0){\put(10,30){\circle*{4}} \put(50,30){\circle*{4}}
\put(30,50){\circle*{4}} \put(10,30){\line(1,0){40}}
\put(10,30){\line(1,1){20}} \put(50,30){\line(-1,1){20}}
\put(5,22){$c_1$} \put(48,22){$c_2$}}

\put(70,0){\put(10,30){\circle*{4}} \put(30,10){\circle*{4}}
\put(50,30){\circle*{4}} \put(30,50){\circle*{4}}
\put(10,30){\line(1,-1){20}} \put(10,30){\line(1,1){20}}
\put(50,30){\line(-1,1){20}} \put(30,10){\line(0,1){40}}
\put(5,22){$c_1$} \put(48,22){$c_2$}}

\put(130,0){\put(10,30){\circle*{4}} \put(30,10){\circle*{4}}
\put(50,30){\circle*{4}} \put(30,50){\circle*{4}}
\put(10,30){\line(1,1){20}} \put(50,30){\line(-1,-1){20}}
\put(50,30){\line(-1,1){20}} \put(30,10){\line(0,1){40}}
\put(5,22){$c_1$} \put(48,22){$c_2$}}

\put(190,0){\put(10,30){\circle*{4}} \put(30,10){\circle*{4}}
\put(50,30){\circle*{4}} \put(30,50){\circle*{4}}
\put(10,30){\line(1,-1){20}} \put(10,30){\line(1,0){40}}
\put(10,30){\line(1,1){20}} \put(50,30){\line(-1,-1){20}}
\put(50,30){\line(-1,1){20}} \put(30,10){\line(0,1){40}}
\put(5,22){$c_1$} \put(48,22){$c_2$}}

\end{picture}
\caption{All possible constructions of the subgraph $H$}
\label{figureH}
\end{center}
\end{figure}

Now consider $B_i$. We shall show that $B_i$ has a
Hamilton path between $c_i$ and a vertex in $N_{B_i}(K)$
as well. If $|V(B_i)|=3$ or 4, then
$B_i$ has the required property. Now suppose
$|V(B_i)|\geq 5$. Since $B_i$ is 2-connected and $c_i$
is simple in $B_i$, $B'_i=B_i-c_i$ is 2-connected.
For every vertex $v\in V(B'_i)$, we have
$$d_{B'_i}(v)\geq\left\{\begin{array}{ll}
  k-2,  & v\in N_{B'_i}(c_i);\\
  k-1,  & \mbox{otherwise}.
\end{array}\right.$$

For each pair of vertices $u,v\in V(B'_i)$
with distance two in $B'_i$, either $u\notin N_{B'_i}(c_i)$
or $v\notin N_{B'_i}(c_i)$; for otherwise $uv\in E(G)$, a
contradiction. So $$\max\{d_{B'_i}(u),d_{B'_i}(v)\}\geq k-1\geq \frac{|V(B'_i)|+1}{2},$$
unless $(k,|V(B_i)|)\in\{(3,5),(3,6),(4,7)\}$.
Suppose $(k,|V(B_i)|)\notin\{(3,5),(3,6),(4,7)\}$.
Notice $|V(B'_i)|\geq 4$. If $B'_i$ is 3-connected, then by Theorem \ref{ThBW}, $B'_i$ is
Hamiltonian-connected and $B_i$ has a Hamilton path between $c_i$ and a
vertex in $N_{B_i}(K)$. If $B'_i$ is not 3-connected, then consider
two vertices $u$ and $v$ separated by a 2-cut $\{x,y\}$. Without loss of
generality, assume $v\notin N_{B'_i}(c_i)$. In fact, if $k=3$ then
$|V(B_i)|\leq k+3=6$. Recall that $|V(B_i)|\geq 5$.
So we have $(k,|V(B_i)|)\in\{(3,5),(3,6)\}$, a contradiction.
If $k=4$ then $|V(B_i)|\leq k+3=7$, and $(k,|V(B_i)|)\in\{(4,5),(4,6)\}$.
Now suppose that either $k\geq 5$ or $(k,|V(B_i)|)\in\{(4,5),(4,6)\}$. If
$k\geq 5$, then
$d_{B'_i}(u)+d_{B'_i}(v)\geq(k-1)+(k-2)=2k-3\geq k+2\geq |V(B'_i)|$.
If $(k,|V(B_i)|)\in\{(4,5),(4,6)\}$, then
$d_{B'_i}(u)+d_{B'_i}(v)\geq(k-1)+(k-2)=2k-3=5\geq |V(B'_i)|$.
On the other hand, since $\{x,y\}$ is a 2-cut of $B'_i$ and the cut
separates $u$ and $v$, we have $d_{B'_i}(u)+d_{B'_i}(v)\leq |V(B'_i)|$.
Thus, it follows that $d_{B'_i}(u)=k-2$ and $u\in N_{B'_i}(c_i)$. Moreover,
since $u$ and $v$ are chosen arbitrarily,
$B'_i-\{x,y\}$ consists of two cliques and $x,y$ are adjacent
to all vertices in $V(B'_i)\backslash\{x,y\}$.
Finally, consider the leftover case of
$(k,|V(B_i)|)\in\{(3,5),(3,6),(4,7)\}$. Note that $B_i$ is claw-free,
closed, and $d_{B_i}(v)\geq k-1$ for $v\in V(B_i)\backslash\{c_i\}$.
Furthermore, $B_i$ is 2-connected and $N_{B_i}(c_i)$ is complete. Armed
with these properties,
one can check that $B_i$ is either complete or one of the graphs in
Figure \ref{figureB_i}. Moreover, every vertex of degree $k-1$ has a neighbor in
$K$. Thus $B_i$ has the required property.

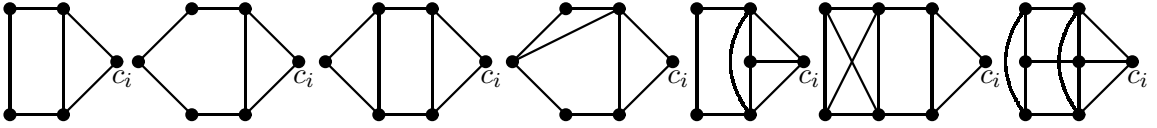
\begin{figure}[h]
\begin{center}
\begin{picture}(465,70)

\thicklines

\put(0,0){\put(10,10){\circle*{4}} \put(10,50){\circle*{4}}
\put(30,10){\circle*{4}} \put(30,50){\circle*{4}}
\put(50,30){\circle*{4}} \put(10,10){\line(1,0){20}}
\put(10,10){\line(0,1){40}} \put(10,50){\line(1,0){20}}
\put(30,10){\line(0,1){40}} \put(50,30){\line(-1,-1){20}}
\put(50,30){\line(-1,1){20}} \put(48,22){$c_i$}}

\put(48,0){\put(10,30){\circle*{4}} \put(30,10){\circle*{4}}
\put(30,50){\circle*{4}} \put(50,10){\circle*{4}}
\put(50,50){\circle*{4}} \put(70,30){\circle*{4}}
\put(10,30){\line(1,-1){20}} \put(10,30){\line(1,1){20}}
\put(30,10){\line(1,0){20}} \put(30,50){\line(1,0){20}}
\put(50,10){\line(0,1){40}} \put(70,30){\line(-1,-1){20}}
\put(70,30){\line(-1,1){20}} \put(68,22){$c_i$}}

\put(118,0){\put(10,30){\circle*{4}} \put(30,10){\circle*{4}}
\put(30,50){\circle*{4}} \put(50,10){\circle*{4}}
\put(50,50){\circle*{4}} \put(70,30){\circle*{4}}
\put(10,30){\line(1,-1){20}} \put(10,30){\line(1,1){20}}
\put(30,10){\line(1,0){20}} \put(30,10){\line(0,1){40}}
\put(30,50){\line(1,0){20}} \put(50,10){\line(0,1){40}}
\put(70,30){\line(-1,-1){20}} \put(70,30){\line(-1,1){20}}
\put(68,22){$c_i$}}

\put(188,0){\put(10,30){\circle*{4}} \put(30,10){\circle*{4}}
\put(30,50){\circle*{4}} \put(50,10){\circle*{4}}
\put(50,50){\circle*{4}} \put(70,30){\circle*{4}}
\put(10,30){\line(1,-1){20}} \put(10,30){\line(1,1){20}}
\put(10,30){\line(2,1){40}} \put(30,10){\line(1,0){20}}
\put(30,50){\line(1,0){20}} \put(50,10){\line(0,1){40}}
\put(70,30){\line(-1,-1){20}} \put(70,30){\line(-1,1){20}}
\put(68,22){$c_i$}}

\put(257,0){\put(10,10){\circle*{4}} \put(10,50){\circle*{4}}
\put(30,10){\circle*{4}} \put(30,30){\circle*{4}}
\put(30,50){\circle*{4}} \put(50,30){\circle*{4}}
\put(10,10){\line(1,0){20}} \put(10,10){\line(0,1){40}}
\put(10,50){\line(1,0){20}} \put(30,10){\line(0,1){40}}
\qbezier(30,10)(15,30)(30,50) \put(50,30){\line(-1,-1){20}}
\put(50,30){\line(-1,0){20}} \put(50,30){\line(-1,1){20}}
\put(48,22){$c_i$}}

\put(305,0){\put(10,10){\circle*{4}} \put(10,50){\circle*{4}}
\put(30,10){\circle*{4}} \put(30,50){\circle*{4}}
\put(50,10){\circle*{4}} \put(50,50){\circle*{4}}
\put(70,30){\circle*{4}} \put(10,10){\line(1,0){40}}
\put(10,10){\line(1,2){20}} \put(10,10){\line(0,1){40}}
\put(10,50){\line(1,0){40}} \put(10,50){\line(1,-2){20}}
\put(30,10){\line(0,1){40}} \put(50,10){\line(0,1){40}}
\put(70,30){\line(-1,-1){20}} \put(70,30){\line(-1,1){20}}
\put(68,22){$c_i$}}

\put(380,0){\put(10,10){\circle*{4}} \put(10,30){\circle*{4}}
\put(10,50){\circle*{4}} \put(30,10){\circle*{4}}
\put(30,30){\circle*{4}} \put(30,50){\circle*{4}}
\put(50,30){\circle*{4}} \put(10,10){\line(1,0){20}}
\put(10,10){\line(0,1){40}} \put(10,50){\line(1,0){20}}
\put(30,10){\line(0,1){40}} \qbezier(10,10)(-5,30)(10,50)
\qbezier(30,10)(15,30)(30,50) \put(50,30){\line(-1,-1){20}}
\put(50,30){\line(-1,0){40}} \put(50,30){\line(-1,1){20}}
\put(48,22){$c_i$}}

\end{picture}
\caption{All possible constructions of the subgraph $B_i$}
\label{figureB_i}
\end{center}
\end{figure}

This implies that there is a perfect path $P$ corresponding to $G'$.
Together with a Hamilton path of $G[K]$, one can find a Hamilton cycle
of $G$, a contradiction.

\begin{case}
  $G'$ is 2-connected.
\end{case}

Note that $\delta(G')\geq k-1\geq \frac{2k}{3}\geq \frac{|V(G')|-2}{3}$. By
Theorem \ref{ThMS}, $G'$ is Hamiltonian. Let $C$ be a
Hamilton cycle of $G'$ with a given orientation.

We claim $|N_K(G')|\leq k$. Suppose $|N_K(G')|\geq k+1$.
Recall that every vertex in $G'$ has at most
one neighbor in $K$, and $N_{G'}(x)$ is a clique for every vertex
$x\in N_K(G')$. If there are two successive vertices $u$ and $v$
on $C$ that have distinct neighbors in $K$, then there will be a
perfect path corresponding to $G'$ and $G$ will be Hamiltonian, a
contradiction. So $|V(G')|\geq 2|N_K(G')|$, which
implies that $|N_K(G')|=k+1$, $|V(G')|=2k+2$, and every vertex in
$N_K(G')$ has exactly one neighbor in $G'$. Let $x\in N_K(G')$ and
$N_{G'}(x)=\{v\}$. Let $v^+,v^-$, and $v^{--}$ be the successor, the predecessor,
and the second predecessor, of $v$, respectively. Then $v^{--}$ has a
neighbor $y$ in $K$. We have $v^+v^-\in E(G)$; for otherwise
$\{v,x,v^+,v^-\}$ induces a claw. Now $P=xvv^-C[v^+,v^{--}]y$ is a
perfect path corresponding to $G'$, and $G$ is Hamiltonian, a
contradiction. Thus $|N_K(G')|\leq k$.

Let $K'\subseteq K$ such that $N_K(G')\subseteq K'$ and
$|K'|+|V(G')|=3k+3$. Now consider the graph $G''=G[V(G')\cup K']$.
We claim that $\delta(G'')\geq k$. Indeed, for each $v\in V(G')$,
all its neighbors in $G$ are in $G''$, consequently, $d_{G''}(v)\geq k$.
Moreover, recall that $|V(G')|\leq 2k+2$, and this implies $|K'|\geq k+1$.
Since $K'$ is a clique in $G''$, for each vertex $v\in K'$, $d_{G''}(v)\geq d_{G''[K']}(v)\geq k$.
Thus $\delta(G'')\geq k$. Obviously, $G''$ is 2-connected and $|V(G'')|=3k+3\geq 12$
since $k\geq 3$. As $\delta(G'')\geq k=\frac{|V(G'')|-3}{3}\geq \frac{|V(G'')|}{4}$
when $|V(G'')|\geq 12$, by Theorem \ref{ThLi}, $G''\in \mathcal{F}$ or
$G''$ is Hamiltonian.

Recall that $|K'|\geq k+1$. So $K'\backslash (N_K(G'))\neq \emptyset$
and $K'$ contains at least one simple vertex of $G$, say $v$.
Suppose that $G''$ is Hamiltonian.
Let $C''$ be a Hamilton cycle of $G''$. Then the predecessor and successor
of $v$ on $C''$ are all in $K$. Thus $C''$ can
be extended to a Hamilton cycle of $G$, a contradiction. Now suppose $G''\in \mathcal{F}$.
Let $C_1,C_2,C_3$ be the three disjoint maximal cliques of $G''$
and let $v\in V(C_1)$. Then one can see that $C_1\subset K$, $K\cap C_2=\emptyset$,
and $K\cap C_3=\emptyset$. So it is obvious $\alpha(G)=3$,
a contradiction to Claim 1.
The proof is complete. \hfill $\Box$

\smallskip
\noindent
{\bf Proof of Theorem \ref{ThStruClaw-Trac}.} Suppose that $G$ is not
traceable. Since $e(G)\geq e(EN_n)=\binom{n-3}{2}+3\geq \frac{n(n-9)}{2}+21$
when $n\geq 12$, we get $G\subseteq EN_n$ by Lemma \ref{LeeGENn}.
As $e(G)\geq e(EN_n)$, we have $G=EN_n$. {\hfill$\Box$}

\section{ Spectral inequalities}
The main purpose of this section is to compare the largest
eigenvalues of the adjacency matrices (and the signless Laplacian
matrices) of a few related graphs. Before stating our results,
let us recall some results for the largest eigenvalue of the
adjacency matrix of a graph.
\begin{theorem}[Hong, Shu, and Fang \cite{HSK}]\label{ThHoShKa}
Let $G$ be a connected graph with $n$ vertices and $m$ edges. If the minimum
degree $\delta(G)\geq k$, then
$$\mu(G)\leq\frac{k-1+\sqrt{(k+1)^2+4(2m-kn)}}{2}.$$
\end{theorem}

\begin{theorem}[Hofmeister \cite{H88}]\label{ThHo}
Let $G$ be a graph. Then $$\mu(G)\geq\sqrt{\frac{\sum_{v\in
V(G)}d^2(v)}{n}}.$$
\end{theorem}

We also need the following upper bound on $q(G)$.
\begin{theorem}[Feng and Yu \cite{FY}]\label{ThFeYu}
Let $G$ be a graph on $n$ vertices and $m$ edges. Then $$q(G)\leq
\frac{2m}{n-1}+n-2.$$
\end{theorem}
Proofs of Theorems \ref{ThmuG}, \ref{ThmuCG},
\ref{ThqG}, and \ref{ThTraceable}
rely heavily on the following lemma.
\setcounter{lemma}{3}
\begin{lemma}\label{LeCompare}
For $n\geq 10$, we have \\
(1) $\mu(EB_n)>\mu(EB'_n)>\mu(K_{n-6})=n-7$;\\
(2) $q(EB_n)>q(EB'_n)>q(K_{n-6})=2n-14$;\\
(3)
$\mu(\overline{EB'_n})<\mu(\overline{EB_n})<\mu(K_6\vee(n-6)K_1)=(5+\sqrt{24n-119})/2.$
\end{lemma}

\begin{proof}
Since proofs of Inequalities $(1)$, $(2)$, and $(3)$ are quite similar, we will give
details for the proof of  inequality $(1)$ here and sketch  others.

\noindent
{\it Proof of Inequality (1):}  Let $\mu_1$  (resp.,~$\mu_2$) be the largest eigenvalue
of the adjacency matrix of $EB_n$  (resp.,~$EB'_n$).  We will compute the characteristic
equations of their adjacency matrices directly. Let $ x$ (resp.,~$ y$) be the eigenvector
corresponding to $\mu_1$ (resp.,~$\mu_2$). We observe that $EB_n$  has only four types of
vertices with respect to their degrees, namely $d(v) \in \{2, 4, n-7, n-5\}$
for each $v \in V(EB_n)$. Furthermore, if two vertices $u$ and $v$ have the
same degree, then ${x}_u={x}_v$ by symmetry. Similar observations also
hold for the graph $EB_n'$.

Let $u$ be a vertex of degree two in $EB_n$, $v$ a vertex of degree
four in $EB_n$, $w$ a vertex of degree $n-5$ in $EB_n$, and $z$ a
vertex of degree $n-7$ in $EB_n$. If $A$ is the adjacency matrix of
$EB_n$, then we have
\[
A{x}=\mu_1 { x}.
\]
In particular, we get $(A{x})_p=\mu_1 { x}_p$ for $p \in \{u,v,w,z\}$.
Recall the observation above. We get the following system of linear equations:
\begin{align}
x_v+x_w&=\mu_1 x_u  \label{eq11}\\
x_u+2x_v+x_w&= \mu_1 x_v  \label{eq12}\\
x_u+x_v+2x_w+(n-9)x_z&=\mu_1 x_w \label{eq13}\\
3x_w+(n-10)x_z&=\mu_1 x_z \label{eq14}.
\end{align}
Viewing $x_w$ as a free variable and solving for $x_u$ and $x_v$ from \eqref{eq11} and
\eqref{eq12}, we get
\begin{align}
x_u&=\frac{\mu_1-1}{\mu_1^2-2\mu_1-1} x_w \label{eq17}\\
x_v&=\frac{\mu_1+1}{\mu_1^2-2\mu_1-1} x_w \label{eq16}.
\end{align}
Solving for $x_z$  from \eqref{eq14}, we get
\begin{equation} \label{eq18}
x_z=\frac{3}{\mu_1-n+10} x_w.
\end{equation}
Putting \eqref{eq16}, \eqref{eq17},  and \eqref{eq18} in \eqref{eq13}, we obtain
\[
\frac{\mu_1-1}{\mu_1^2-2\mu_1-1}x_w+\frac{\mu_1+1}{\mu_1^2-2\mu_1-1}x_w+2x_w+\frac{3(n-9)}{\mu_1-n+10}x_w=\mu_1x_w.
\]
The Perron-Frobenius theorem implies that all entries of $x$ are positive. So we
can cancel $x_w$ from both sides of the above equation. Simplify the resulting
equation. We get the characteristic equation
\begin{equation} \label{poly11}
(\mu_1-n+10)(\mu_1^3-4\mu_1^2+\mu_1+2)-(3n-27)(\mu_1^2-2\mu_1-1)=0.
\end{equation}
For $EB_n'$, using the same idea, we get a new system of linear equations:
\begin{align}
2y_v&=\mu_2 y_u \label{eq21}\\
y_w+2y_v+y_u&=\mu_2 y_v \label{eq22} \\
2y_v+y_w+(n-8)y_z&=\mu_2 y_w \label{eq23}\\
2y_w+(n-9)y_z&=\mu_2y_z \label{eq24}.
\end{align}
From \eqref{eq21}, \eqref{eq22}, and \eqref{eq23}, we get
\begin{align}
y_v&=\frac{\mu_2}{2} y_u  \nonumber\\
y_w&=(\tfrac{\mu_2^2}{2}-\mu_2-1) y_u \label{eq25} \\
y_z&=\frac{(\mu_2^2-2\mu_2-2)(\mu_2-1)-2\mu_2}{2n-16} y_u \label{eq26}.
\end{align}
\eqref{eq25}, together with \eqref{eq24}, tells us
\begin{equation} \label{eq27}
y_z=\frac{\mu_2^2-2\mu_2-2}{\mu_2-n+9} y_u.
\end{equation}
Recall $y_u >0$. Equalizing \eqref{eq26} and \eqref{eq27} followed by cancelling $y_u$,  we get
\begin{equation} \label{poly12}
(\mu_2-n+9)\left( (\mu_2^2-2\mu_2-2)(\mu_2-1) -2\mu_2\right)-(2n-16)(\mu_2^2-2\mu_2-2)=0.
\end{equation}
Define functions
\begin{align}
f(x)&=(x-n+10)(x^3-4x^2+x+2)-(3n-27)(x^2-2x-1),\\
g(x)&=(x-n+9)\left( (x^2-2x-2)(x-1) -2x\right)-(2n-16)(x^2-2x-2).
\end{align}
Let $s=n-7+\tfrac{4}{(n-7)^2}$ and $t=n-7+\tfrac{7}{(n-7)^2}$. Tedious calculus
together with the Maple program can confirm $f(t)<0$ and $g(s)<0<g(t) $ for all $n\geq 12$.
Since $K_{n-6}$ is a subgraph of $EB_n$ and $EB_n'$, we get $\mu_1,\mu_2 \geq n-7$.
By Theorem \ref{ThHoShKa}, we get $\mu_1, \mu_2 < n-5$.  By considering
the largest eigenvalue of $A(EB_n)-\mu_1 I_n$ and $A(EB'_n)-\mu_2 I_n$, we
get the gap between the largest eigenvalue and the second largest eigenvalue of the adjacency
matrix of both $EB_n$ and $EB_n'$ is at least 2. Therefore, for each $n \geq 12$, we have
\[
\mu_1 > t \text{ and } s < \mu_2 < t,
\]
which proves Inequality (1)  for $n \geq 12$.  For $n=10, 11$, we can check Inequality (1)
directly by using Maple program.

\noindent
{\it Proof of Inequality (2):} Let $f(x)$ (resp., $g(x)$) be the characteristic
polynomial of the signless Laplace matrix of $EB_n$ (resp., $EB_n'$).
Using the same idea as proving Inequality (1) (details will be given
in Appendix A), we get
\begin{align*}
f(x)&=(x-2n+17)\left((x-n+3)(x^2-8x+11)-(2x-6)\right)-(3n-27)(x^2-8x+11)\\
g(x)&=(x^2-8x+12)\left((x-2n+16)(x-n+4)- (2n-16)  \right)\\
       &-(2x-4n+32)(2x-n+2)+4n-32.
\end{align*}
We choose $s=2(n-7)+\tfrac{4}{n}$ and $t=2(n-7)+\tfrac{6}{n-7}$. Using
basic calculus (under the help of the Maple program), we can verify $f(t)<0$ and
$g(s)<0<g(t)$ for all $n \geq 27$, which proves the inequality (2) for
$n\geq 27$. For $10 \leq n \leq 26$, we can confirm the inequality using
the Maple program directly.

\noindent
{\it Proof of Inequality (3):}  We use $f(x)$ (resp., $g(x)$) to denote
the characteristic polynomial of the adjacency matrix of $\overline{EB_n}$
(resp., $\overline{EB'_n}$). We can obtain the formula for $f(x)$ and $g(x)$
by the same argument as we did for proving Inequality (1) (the details will
be presented in Appendix B). We get the following
\begin{align*}
f(x)&=x^3-(2x^2+12x+8)-(6n-54)(x+1)\\
g(x)&=(x^2-4n+32)(x^2-2)-(2x^2+x)(x+2)-(2n-16)(x^2+x+2).
\end{align*}
Let $s=\sqrt{6(n-6)}$ and $t=\sqrt{6(n-6)}+1.3$. With the assistance of
the Maple program, we can show $f(t)<0$ and $g(s)<0<g(t)$ for all $n\geq 55$ using
calculus. We have proved the inequality (3) for $n\geq 55$. For the case of $10\leq n\leq 54$,
we can verify Inequality (3) using the Maple program straightforwardly.
\end{proof}

\section{Proofs of  Theorems \ref{ThmuG}, \ref{ThmuCG}, \ref{ThqG},
and \ref{ThTraceable}}
In this section, we prove Theorems \ref{ThmuG},
\ref{ThmuCG}, \ref{ThqG},  and \ref{ThTraceable}, respectively.

\smallskip
\noindent
\textbf{Proof of Theorem \ref{ThmuG}.}
By Lemma \ref{LeCompare} and Theorem \ref{ThHoShKa}, we have
$$n-7\leq\mu(G)\leq\frac{1+\sqrt{9+4(2e(G)-2n)}}{2}.$$
One can get
$$e(G)\geq\frac{n(n-13)}{2}+27\geq\frac{n(n-15)}{2}+57.$$
(Here we used the assumption $n\geq 30$). By Lemma \ref{LeeGEBn},
either $G$ is Hamiltonian, or $G\subseteq
EB_n$, or $G\subseteq EB'_n$. However, if either
$G \subsetneq EB_n$ or $G\subseteq EB'_n$, then
by Lemma \ref{LeCompare}(i), $\mu(G)<\mu(EB_n)$,
which leads to a contradiction. Thus  $G$ is
Hamiltonian unless $G=EB_n$.
{\hfill$\Box$}

\smallskip
\noindent
\textbf{Proof of Theorem \ref{ThmuCG}.} The proof is inspired
by \cite{FN,Z}. Let $G'=cl(G)$. If $G'$ is
Hamiltonian, then $G$ is also Hamiltonian by Theorem
\ref{ThR}. Now we assume $G'$ is not Hamiltonian. From
Lemma \ref{Ledudv}, it follows that the degree sum of
every two nonadjacent vertices $u$ and $v$ in $G'$ is
at most $n-1$. This gives
$$d_{\overline{G'}}(u)+d_{\overline{G'}}(v)\geq 2(n-1)-(n-1)=n-1.$$
By Theorem \ref{ThHo}, we have
$$\mu(\overline{G})\geq\mu(\overline{G'})\geq\sqrt{\frac{\sum_{v\in
V(G)}d_{\overline{G'}}^2(v)}{n}}=\sqrt{\frac{\sum_{uv\in
E(\overline{G'})}(d_{\overline{G'}}(u)+d_{\overline{G'}}(v))}{n}}
\geq\sqrt{\frac{(n-1)e(\overline{G'})}{n}}.$$
 Lemma \ref{LeCompare} gives
$$\frac{5+\sqrt{24n-119}}{2}\geq\mu(\overline{G})\geq\sqrt{\frac{(n-1)e(\overline{G'})}{n}}.$$
One can get
$$e(G')={n \choose 2}-e(\overline{G'})\geq {n\choose2}-\left(\frac{5+\sqrt{24n-119}}{2}\right)^2\cdot\frac{n}{n-1}>\frac{n(n-15)}{2}+56,$$
where the condition ``$n\geq 219$" is used for the last inequality. By Lemma \ref{LeeGEBn},
either $G\subseteq G'\subseteq EB_n$ or $G\subseteq G'\subseteq EB'_n$.
However, if $G\subseteq EB_n$ or $G\subsetneq EB'_n$, then
$\mu(\overline{G})>\mu(\overline{EB'_n})$, which is a contradiction.
Thus $G=EB'_n$. \hfill $\Box$

\smallskip
\noindent
\textbf{Proof of Theorem \ref{ThqG}.} From
Lemma \ref{LeCompare} and Theorem \ref{ThFeYu}, we have
$$2n-14\leq q(G)\leq\frac{2e(G)}{n-1}+n-2.$$
One can get
$$e(G)\geq\frac{(n-1)(n-12)}{2}\geq\frac{n(n-15)}{2}+57.$$
(The last inequality is true because $n\geq 51$). By Lemma \ref{LeeGEBn},
either $G$ is Hamiltonian or $G\subseteq EB_n$ or $G\subseteq EB'_n$.
However, if $G\subsetneq EB_n$ or $G\subseteq EB'_n$, then
$q(G)<q(EB_n)$, a contradiction. Thus $G=EB_n$.
{\hfill$\Box$}

\smallskip
\noindent
\textbf{Proof of Theorem \ref{ThTraceable}.} Note that $q(EN_n)\geq
q(K_{n-3})\geq 2n-8$. By Theorem \ref{ThFeYu}, we have
$$2n-8\leq q(G)\leq
\frac{2e(G)}{n-1}+n-2.$$
We get
$$e(G)\geq\frac{(n-1)(n-6)}{2}\geq\frac{n(n-9)}{2}+21,$$
where the assumption $n\geq 18$ is used for the second inequality. By Lemma \ref{LeeGENn}, either
$G$ is traceable or $G\subseteq EN_n$. If $G\subsetneq EN_n$, then $q(G)<q(EN_n)$,
which gives a contradiction. Thus $G=EN_n$.
{\hfill$\Box$}

\section{A concluding remark}
Recently, Li and Ning \cite{LN16}, and independently, F\"{u}redi,
Kostochka, and Luo \cite{FKL} obtained versions of stability theorems of Erd\H{o}s'
theorem \cite{E}, respectively. We refer the interested reader
to F\"{u}redi, Kostochka, and Luo \cite{FKL-arxiv} for more developments in this direction.
It may be interesting to find stability versions of Theorem \ref{ThgeneralErdos}.

\section*{Acknowledgments}
B.-L. Li was supported by the NSFC grant (No.\ 11601429).
B. Ning was supported by the NSFC grants (No.\ 11601379
and No.\ 11771141) and the Seed Foundation of Tianjin University (2018XRG-0025). 
X. Peng was supported by the NSFC
grant (No.\ 11601380). The authors are very grateful to two
anonymous referees for carefully reading the manuscript
and for giving valuable comments which helped improving
the presentation of the paper. The revised version was made
when the second author was visiting Professor Fengming
Dong at NTU, Singapore. The second author is very grateful
to all his help, enthusiasm, and encouragement during
his visit.

\section{Appendix}
\smallskip
\noindent
{\bf Appendix A:}  Let $q_1$  (resp.,~$q_2$) be the largest eigenvalue of
the signless Laplacian matrix of $EB_n$ (resp.,~$EB'_n$).
Let $x$ (resp.,~$y$) be the eigenvector corresponding to $q_1$ (resp.,~$q_2$).
Pick $u$ as a vertex of degree two in $EB_n$, $v$ as a vertex of degree four
in $EB_n$, $w$ as a vertex of degree $n-5$ in $EB_n$, and $z$ as a vertex of
degree $n-7$ in $EB_n$. If $Q$ is the signless Laplacian matrix of $EB_n$, then we have
$Qx=q_1x$.  By symmetry, we get the
following system of linear equations:
\begin{align}
2x_u+x_v+x_w&=q_1x_u \label{eq31}\\
6x_v+x_u+x_w&=q_1x_v \label{eq32}\\
(n-3)x_w+x_u+x_v+(n-9)x_z&=q_1x_w \label{eq33}\\
(2n-17)x_z+3x_w&=q_1x_z. \label{eq34}
\end{align}
 We view $x_w$ as a variable and solve for $x_u$, $x_v$, and $x_z$
 from \eqref{eq31}, \eqref{eq32} and \eqref{eq34}. We get
 \begin{align}
 x_u&=\frac{q_1-5}{q_1^2-8q_1+11} x_w \label{eq35}\\
 x_v&=\frac{q_1-1}{q_1^2-8q_1+11} x_w \label{eq36} \\
 x_z&=\frac{3}{q_1-2n+17}x_w. \label{eq37}
  \end{align}
Putting \eqref{eq35}, \eqref{eq36}, and \eqref{eq37} in \eqref{eq33}, we get
\[
\frac{q_1-5}{q_1^2-8q_1+11} x_w+\frac{q_1-1}{q_1^2-8q_1+11} x_w+\frac{3n-27}{q_1-2n+17}x_w=(q_1-n+3)x_w.
\]
Cancel $x_w$ from both sides of the  equation above and simplify it.
We get $q_1$ must satisfy the following equation
\[
(3n-27)(q_1^2-8q_1+11)=(q_1-2n+17)((q_1-n+3)(q_1^2-8q_1+11)-(2q_1-6)),
\]
which gives the characteristic polynomial of $Q(EB_n)$.

For $EB_n'$, we select vertices $u,v,w$ and $z$ similarly.
If $Q$ is the signless Laplacian matrix of $EB_n'$, then we
have $Qy=q_2y$. We obtain the following system of inequalities:
\begin{align}
2y_u+2y_v&=q_2y_u  \label{eq41} \\
6y_v+y_u+y_w&=q_2y_v \label{eq42} \\
(n-4)y_w+2y_v+(n-8)y_z&=q_2 y_w \label{eq43} \\
(2n-16)y_z+2y_w&=q_2y_z\label{eq44}.
\end{align}
Using  $y_u$ to express $y_v$, $y_w$, and $y_z$ in \eqref{eq41},
\eqref{eq42}, and \eqref{eq43},  we get
\begin{align}
y_v&=\frac{q_2-2}{2} y_u  \label{eq45} \\
y_w&=\left(\frac{1}{2}(q_2-2)(q_2-6) -1 \right) y_u  \label{eq46}\\
y_z&=\frac{1}{n-8} \left( (q_2-n+4) \left( \frac{1}{2}(q_2-2)(q_2-6)-1 \right)-(q_2-2) \right) y_u. \label{eq47}
\end{align}
We bring \eqref{eq45}, \eqref{eq46}, and \eqref{eq47} into \eqref{eq44}.  Writing $A=(q_2-2)(q_2-6)$, we  get
\[
\frac{(q_2-2n+16)}{n-8}((q_2-n+4)(\tfrac{1}{2}A-1)-(q_2-2))y_u=(A-2)y_u.
\]
 Cancelling $y_u$ and multiplying by $2(n-8)$, we get
 \[
(q_2-2n+16)((q_2-n+4)(A-2)-(2q_2-4))=(A-2)(2n-16).
\]
Expanding the left side of the equation above and simplifying it, we get
\[
(q_2-2n+16)(q_2-n+4)A-2(q_2-2n+16)(2q_2-n+2)=(2n-16)A-(4n-32).
\]
Thus $q_2$ must
satisfy the following equation
\[
(q_2^2-8q_2+12)\left((q_2-2n+16)(q_2-n+4)- (2n-16)  \right)-(2q_2-4n+32)(2q_2-n+2)+4n-32=0,
\]
which gives the characteristic polynomial of $E(EB_n')$.

\vspace{0.3cm}

\noindent
{\bf Appendix B:} For $\overline{EB_n}$, let $v$ be a vertex of
degree $n-5$, $u$ a vertex of degree $n-3$, $w$ a vertex
of degree $4$, and $z$ a vertex of degree $6$.  We use
$\mu_1$ to  denote the largest eigenvalue of  the adjacency
matrix of $\overline{EB_n}$. If $x$ is the eigenvector of $\mu_1$, then we
have $Ax=\mu_1x$.  Recall the symmetry between the vertices with
the same degree. We get
\begin{align}
3x_u+3x_v&=\mu_1x_z \label{eq51} \\
2x_u+2x_v&=\mu_1x_w \label{eq52} \\
2x_u+2x_w+(n-9)x_z&= \mu_1 x_v \label{eq53}\\
2x_u+2x_v+2x_w+(n-9)x_z&=\mu_1 x_u. \label{eq54}
\end{align}
\eqref{eq51}-\eqref{eq52} give
\begin{equation} \label{eq55}
x_u+x_v=\mu_1x_z-\mu_1x_w.
\end{equation}
\eqref{eq54}-\eqref{eq53} yield
\begin{equation} \label{eq56}
2x_v=\mu_1x_u-\mu_1x_v.
\end{equation}
We form a new system of linear equations using \eqref{eq52},
\eqref{eq55}, and \eqref{eq56}. We solve for $x_u$,  $x_w$,
and $x_z$ from the new system, and get
\begin{align*}
x_u&=\frac{\mu_1+2}{\mu_1} x_v \\
x_w&=\left( \frac{2(\mu_1+2)}{\mu_1^2} +\frac{2}{\mu_1}\right) x_v \\
x_z&=\left( \frac{3(\mu_1+2)}{\mu_1^2} +\frac{3}{\mu_1}\right) x_v.
\end{align*}
We put expressions of $x_u$, $x_w$, and $x_z$ in \eqref{eq53}.
Then we cancel $x_v$ from both sides of the resulting equation
and simplify it. We get
\[
\mu_1^3-(2\mu_1^2+12\mu_1+8)-(6n-54)(\mu_1+1)=0,
\]
which completes the proof.

For $\overline{EB_n'}$, we use the same assumptions as
$\overline{EB_n}$.  Let $\mu_2$ be the largest eigenvalue
of the adjacency matrix of $\overline{EB_n'}$
and $y$ be the corresponding eigenvector. We have the
following system of linear equations: \begin{align}
4y_v+2y_u&=\mu_2y_z \label{eq61}\\
2y_v+2y_u&=\mu_2y_w \label{eq62} \\
y_u+y_v+y_w+(n-8)y_z&=\mu_2y_v \label{eq63}\\
y_u+2y_v+2y_w+(n-8)y_z&=\mu_2y_u. \label{eq64}
\end{align}
\eqref{eq64}-\eqref{eq63} give
\begin{equation} \label{eq65}
y_v+y_w=\mu_2y_u-\mu_2y_v.
\end{equation}
\eqref{eq61}-\eqref{eq62} give
\begin{equation} \label{eq66}
2y_v=\mu_2y_z-\mu_2y_w.
\end{equation}
We solve for $y_u$, $y_w$, and $y_z$ from the system
consisting of Equations \eqref{eq62}, \eqref{eq65},
and \eqref{eq66}. We obtain
\begin{align*}
y_u&=\frac{\mu_2^2+\mu_2+2}{\mu_2^2-2} y_v \\
y_w&=\left( \frac{2}{\mu_2}+\frac{2(\mu_2^2+\mu_2+2)}{\mu_2(\mu_2^2-2)} \right) y_v \\
y_z&=\left( \frac{4}{\mu_2}+\frac{2(\mu_2^2+\mu_2+2)}{\mu_2(\mu_2^2-2)} \right) y_v.
\end{align*}
We bring expressions of $y_u$, $y_w$, and $y_z$ in
\eqref{eq63}, cancel $y_v$, and simplify the equation.
We get $\mu_2$ must satisfy the following equation:
\[
(\mu_2^2-4n+32)(\mu_2^2-2)-(2\mu_2^2+\mu_2)(\mu_2+2)-(2n-16)(\mu_2^2+\mu_2+2)=0,
\]
which completes the proof.
\end{document}